\documentclass[12pt,oneside]{amsart}

\usepackage{amssymb}
\usepackage{amsmath}
\usepackage{amsthm}
\usepackage{amscd}
\usepackage[all]{xy}
\usepackage{longtable}

\theoremstyle{plain}
\newtheorem{theorem}{Theorem}
\newtheorem{lemma}{Lemma}
\newtheorem{corollary}{Corollary}
\newtheorem{proposition}{Proposition}

\theoremstyle{definition}

\theoremstyle{remark}

\begin{document}
\title
[Cohomology of vector bundles and non-pluriharmonic loci]
{Cohomology of vector bundles and non-pluriharmonic loci}
\author{Yusaku Tiba}
\date{}

\begin{abstract}
In this paper, we study cohomology groups of vector bundles on neighborhoods of a non-pluriharmonic locus in Stein manifolds and in projective manifolds.  
By using our results, we show variants of the Lefschetz hyperplane theorem.  

\end{abstract}

\maketitle



\section{Introduction}\label{section:1}
Let $X$ be a Stein manifold.   
Let $\varphi$ be an exhaustive plurisubharmonic function on $X$.  
The support of $i \partial \overline{\partial}\varphi$ has some interesting properties.  
In this paper, we study the cohomology of holomorphic vector bundles on open neighborhoods of $\mathrm{supp}\,i \partial \overline{\partial} \varphi$.  
Here we denote by $\mathrm{supp}\, T$ the support of a current $T$.  
Let $F$ be a holomorphic vector bundle over $X$.  
Let $A \subset X$ be a closed set.  
For any open neighborhoods $V \subset U$ of $A$, the inclusion map induces 
$H^{q}(U, F) \to H^{q}(V, F)$.  
We define the direct limit 
$\underset{A \subset V}{\varinjlim}H^{q}(V, F)$ where $V$ runs through all open neighborhoods of $A$.  
Our first main result is the following: 

\begin{theorem}\label{theorem:1}
Let $X$ be a Stein manifold of dimension $n$ ($n \geq 3$).
Let $m$ be a positive integer which satisfies $1 \leq m \leq n-2$.  
Let $\varphi_{1}, \ldots, \varphi_{m}$ be non-constant plurisubharmonic functions on $X$ such that for every $r < \sup_{X} \varphi_{j}$ the sublevel set $\{z \in X\, |\, \varphi_{j}(z) \leq r\}$ is compact ($1 \leq j \leq m$).  
Let $F$ be a holomorphic vector bundle over $X$.  
Then 
the natural map 
\[
H^{0}(X, F) \to \underset{\bigcap_{j=1}^{m}\mathrm{supp}\, i \partial \overline{\partial}\varphi_{j} \subset V}{\varinjlim}H^{0}(V, F).  
\]
is an isomorphism and
\[
\underset{\bigcap_{j=1}^{m}\mathrm{supp}\, i\partial \overline{\partial} \varphi_{j} \subset V}{\varinjlim} H^{q}(V, F) = 0
\]
for $0 < q < n-m-1$.    
\end{theorem}

Let $X$ be a projective manifold.  
We have the Hodge decompotion 
$H^{2}(X, \mathbb{C}) = H^{2,0}(X, \mathbb{C}) \oplus H^{1,1}(X, \mathbb{C}) \oplus H^{0,2}(X, \mathbb{C})$.
Define $H^{1,1}(X, \mathbb{R}) = H^{1,1}(X, \mathbb{C}) \cap H^{2}(X, \mathbb{R})$.  
Let $\mathcal{K}_{NS} \subset H^{1,1}(X, \mathbb{R})$ be the open cone generated by classes of ample divisors (see Section~6 of \cite{Dem2}).  
Our second main result is the following:  

\begin{theorem}\label{theorem:2}
Let $X$ be a projective manifold of dimension $n$ ($n \geq 3$).
Let $m$ be a positive integer which satisfies $1 \leq m \leq n-2$.  
Let  $T_{1}, \ldots, T_{m}$ be closed positive currents of type $(1, 1)$ on $X$ whose cohomology classes belong to $\mathcal{K}_{NS}$.  
let $F$ be a holomorphic vector bundle over $X$.  
Then the natural map 
\[
H^{q}(X, F) \to \underset{\bigcap_{j=1}^{m}\mathrm{supp}\, T_{j} \subset V}{\varinjlim} H^{q}(V, F) 
\]
is an isomorphism for $0 \leq q < n-m-1$ and is injective for $q = n-m-1$.  
\end{theorem}

Let $t \in H^{0}(X, L)$ be a non-zero holomorphic section of the ample line bundle $L$ over $X$.  
We define the hypersurface $Y = \{ z \in X\,|\, t(z) = 0\}$.  
Let $h_{0}$ be a smooth hermitian metric of $L$ such that $\omega_{0} := 2\pi c_{1}(L, h_{0})$ is K\"ahler form.  
Here $c_{1}(L, h_{0})$ is the Chern form of $L$ associated to $h_{0}$.  
Then $Y = \mathrm{supp}\, (\omega_{0} + i\partial \overline{\partial} \log |t|^{2}_{h_{0}})$.  
Because of the vanishing theorem of cohomology groups with compact supports in the Stein manifold $X \setminus Y$, 
we have that the natural map 
\[
H^{q}(X, F) \to \underset{Y  \subset V}{\varinjlim} H^{q}(V, F) 
\]
is an isomorphism for $q < n-1$ and is injective for $q = n-1$.  
Theorem~\ref{theorem:2} is a counterpart of this result.  
Unfortunately, 
we do not know whether Theorem~\ref{theorem:1} and Theorem~\ref{theorem:2} hold in the case when the degree is $n-m$.  

Let $T^{*}_{X}$ be the holomorphic cotangent bundle over $X$.  
If we take $F = \bigwedge^{p}T^{*}_{X}$, our main results imply the following variants of the Lefschetz hyperplane theorem (see Lemma~1 of \cite{Tib18}).  

\begin{corollary}\label{corollary:1}
Let $X, \varphi_{1}, \ldots, \varphi_{m}$ be as in Theorem~\ref{theorem:1}.  
Then the natural map 
\[
H^{q}(X, \mathbb{C}) \to \underset{\bigcap_{j=1}^{m}\mathrm{supp}\, (i\partial \overline{\partial} \varphi_{j}) \subset V}{\varinjlim} H^{q}(V, \mathbb{C}) 
\]
is an isomorphism for $q < n-m-1$ and is injective for $q = n-m-1$.  
\end{corollary}

\begin{corollary}\label{corollary:2}
Let $X, T_{1}, \ldots, T_{m}$ be as in Theorem~\ref{theorem:2}.  
Then the natural map 
\[
H^{q}(X, \mathbb{C}) \to \underset{\bigcap_{j=1}^{m}\mathrm{supp}\, T_{j} \subset V}{\varinjlim} H^{q}(V, \mathbb{C}) 
\]
is an isomorphism for $q < n-m-1$ and is injective for $q = n-m-1$.  
\end{corollary}

Corollary~\ref{corollary:1} generalizes the main theorem of \cite{Tib18}.  
The degree which appears in \cite{Tib18} is $\max\{n-4, 1\}$.  
On the other hand, those which appear in our main results are $n-2$ when $m = 1$.  
The improvement of the degree is due to the method of Lee and Nagata (\cite{Lee-Nag}) and the estimate of the Sobolev norm.  

We prove the case $m=1$ of Theorem~\ref{theorem:1}, \ref{theorem:2} in Section~3, 4, 5.  
By using Mayer-Vietoris sequence, we prove the general case in Section~6.  
\medskip 

{\it Acknowledgment.}
The author would like to thank Seungjae Lee and Yoshikazu Nagata for sending him their paper~\cite{Lee-Nag}.  
This work was supported by the 
Grant-in-Aid for Scientific Research (KAKENHI No.\! 17K14200).

\section{Preliminaries}\label{section:2}
Let $X$ be a K\"ahler manifold and let $\omega$ be a K\"ahler metric on $X$.  
We assume that $X$ is weakly pseudoconvex, that is, there exists a smooth plurisubharmonic exhaustion function on $X$.  
Let $F$ be a holomorphic vector bundle over $X$  
and let $H$ be a smooth hermitian metric of $F$.  
We denote by $L^{(p, q)}(X, F, H, \omega)$ the Hilbert space of $F$-valued $(p, q)$-forms $u$ which satisfy 
\[
\|u\|^{2}_{H, \omega} = \int_{X} |u|^{2}_{H, \omega} dV_{\omega} < +\infty.  
\]
Here $dV_{\omega} = \frac{\omega^{n}}{n!}$.  
Let $i \Theta(F, H)$ be the Chern curvature tensor of $(F, H)$ and let $\Lambda$ be the adjoint of multiplication of $\omega$.  
Suppose that the operator $[i \Theta (F, H), \Lambda]$ acting on $(n, q)$-forms with values in $F$ is positive definite on $X$ ($q \geq 1$).  
Then, for any $\overline{\partial}$-closed form $u \in L^{(n, q)}(X, F, H, \omega)$ which satisfies $\int_{X}\langle [i \Theta(F, H), \Lambda]^{-1}u, u \rangle_{H, \omega} dV_{\omega} < +\infty$, 
there exists $v \in L^{(n, q-1)}(X, F, H, \omega)$ such that $\overline{\partial}v = u$ and that 
\[
\|v\|^{2}_{H, \omega} \leq \int_{X} \langle [i \Theta(F, H), \Lambda]^{-1}u, u \rangle_{H, \omega} dV_{\omega} 
\]
(cf.\,\cite{Dem}).  
We note that $\omega$ is possibly non complete.  

\section{$L^{2}$-estimate}\label{section:3}
In \cite{Tib18}, the surjectivity between the cohomology groups was proved by the Donnelly-Fefferman-Berndtsson type $L^{2}$-estimates for $(0, q)$-forms (\cite{Ber-Cha}, \cite{Don-Fef}).  
In \cite{Lee-Nag}, Lee and Nagata showed that $L^{2}$-Serre duality and $L^{2}$-estimates for not $(0, q)$ but $(n, q)$-forms improve the integrability condition of an $L^{2}$-estimate.  
By using the method of \cite{Lee-Nag}, we prove Proposition~\ref{proposition:1} below.  

Let $X$ be a Stein manifold of dimension $n$ and let $D$ be a relatively compact subdomain in $X$.  
Assume that there exist negative plurisubharmonic functions $\varphi, \eta \in C^{\infty}(\overline{D})$ on $D$ such that $\max\{ \varphi(z), \eta(z)\} \to 0$ when $z \to \partial D$.  
We assume that $\inf_{D} \eta < -1$.  
Define $\phi = - \log (-\varphi)$ and $\rho = \max_{\varepsilon}\{-\log(-\eta), 0\}$.  
Here $\varepsilon > 0$ is a small positive number and $\max_{\varepsilon}$ is a regularized max function (see Chapter~I, Section~5 of \cite{Dem}).  
Let $F$ be a holomorphic vector bundle over $X$ and $H$ be a smooth hermitian metric of $F$.  
We define $F^{*}, H^{*}$ to be the dual of $F, H$.  
Let $\psi \in C^{\infty}(\overline{D})$ be a strictly plurisubharmonic function.  
We take a large positive integer $N$ such that the hermitian vector bundle $(F^{*}, H^{*}e^{-(N-1)\psi})$ is Nakano positive on $\overline{D}$.  
Let $\delta >0$ be a positive number.  
Put $\omega = i\partial \overline{\partial}(\frac{\psi}{\delta} + \frac{\rho}{\delta} + \phi)$.  
Then $\omega$ is a complete K\"ahler metric on $D$.  
Let $\kappa \in C^{\infty}(\mathbb{R})$ such that $\kappa'(t) \geq 1$, $\kappa''(t) \geq 0$ for $t \geq 0$.  
Put $\xi = N\psi + \kappa \circ \rho - \delta \phi$.  
The proof of the following lemma is completely similar to that of Lemma~3.1 of \cite{Lee-Nag}.  
\begin{lemma}\label{lemma:1} 
Let $f \in L^{(n, q)}(D, F^{*}, H^{*}e^{-\xi}, \omega)$ such that $\overline{\partial} f = 0$.  
Assume that $\delta < q$.  
Then there exists $u \in L^{(n, q-1)}(D, F^{*}, H^{*}e^{-\xi}, \omega)$ such that $\overline{\partial}u = f$ and that 
\[
\int_{D}|u|^{2}_{H^{*},\, \omega}e^{-\xi}dV_{\omega} \leq C_{q, \delta} \int_{D}|f|^{2}_{H^{*},\, \omega}e^{-\xi}dV_{\omega}.  
\]
Here $C_{q, \delta}$ depends only on $q$ and $\delta$.  
\end{lemma}
\begin{proof}
There exist relatively compact weakly pseudoconvex subdomains $D_{1} \subset D_{2} \subset \cdots \subset D$ which exhaust $D$.  
Because of the Nakano positivity of $(F^{*}, H^{*}e^{-(N-1)\psi})$, 
there exists 
$u_{k}$ in 
$L^{(n, q-1)}(D_{k}, F^{*}, H^{*}e^{-N\psi - \kappa \circ \rho})$ such that $\overline{\partial} u_{k} = f$ and that 
\[
\int_{D_{k}} |u_{k}|^{2}_{H^{*}, \, \omega}e^{-N\psi - \kappa \circ \rho}dV_{\omega} \leq \int_{D_{k}} \langle [i \Theta(F^{*}, H^{*}e^{-N\psi - \kappa \circ \rho}), \Lambda]^{-1}f, f \rangle_{H^{*}, \, \omega}e^{-N\psi - \kappa \circ \rho}dV_{\omega}.  
\]
Since $\phi$ and $\overline{\partial}\phi$ are bounded in $D_{k}$, 
we have that $u_{k}e^{\delta \phi}$ is the minimal solution of $\overline{\partial} (u_{k}e^{\delta \phi}) = (f + \delta \overline{\partial}\phi \wedge u_{k})e^{\delta \phi}$ in 
$L^{(n, q-1)}(D_{k}, F^{*}, H^{*}e^{-N\psi - \kappa \circ \rho -\delta\phi}, \omega)$. 
Then 
\begin{align*}
& \int_{D_{k}}|u_{k}|^{2}_{H^{*}, \, \omega}e^{-\xi} dV_{\omega}
= \int_{D_{k}}|u_{k}e^{\delta \phi}|^{2}_{H^{*}, \, \omega}e^{-N\psi - \kappa \circ \rho -\delta \phi}dV_{\omega} \\
\leq & 
\int_{D_{k}} \langle [i \Theta(F^{*}, H^{*}e^{-N\psi - \kappa \circ \rho -\delta \phi}), \Lambda]^{-1}(f+\delta \overline{\partial}\phi \wedge u_{k}), f+ \delta \overline{\partial}\phi \wedge u_{k} \rangle_{H^{*}, \, \omega} e^{-\xi}dV_{\omega}  \\
\leq & 
\left(1 + \frac{1}{t}\right) \int_{D_{k}} \langle [i \Theta(F^{*}, H^{*}e^{-N\psi - \kappa \circ \rho -\delta \phi}), \Lambda]^{-1}f, f \rangle_{H^{*}, \, \omega} e^{-\xi}dV_{\omega} \\ 
& +  
(1 + t) \delta^{2}
\int_{D_{k}} \langle [i \Theta(F^{*}, H^{*}e^{-N\psi - \kappa \circ \rho - \delta \phi}), \Lambda]^{-1} \overline{\partial}\phi \wedge u_{k}, \overline{\partial}\phi \wedge u_{k} \rangle_{H^{*}, \, \omega} e^{-\xi}dV_{\omega} 
\end{align*}
for every $t > 0$.  
We have that $\langle [i\Theta(F^{*}, H^{*}e^{-N\psi - \kappa \circ \rho - \delta \phi}), \Lambda]v, v \rangle_{H^{*}, \, \omega} \geq q \delta |v|^{2}_{H^{*}, \, \omega}$ for any $F$-valued $(n, q)$-form $v$ since $i \partial \overline{\partial} (\psi + \kappa \circ \rho + \delta \phi) \geq \delta \omega$.  
Because $|\overline{\partial}\phi|_{\omega} \leq 1$, we have that 
\[
(1+t)\delta^{2}\langle [i \Theta(F^{*}, H^{*}e^{-N\psi - \kappa \circ \rho - \delta \phi}), \Lambda]^{-1} \overline{\partial}\phi \wedge u_{k}, \overline{\partial}\phi \wedge u_{k} \rangle_{H^{*}, \, \omega} \leq (1 +t) \frac{\delta}{q}|u_{k}|^{2}_{H^{*},\, \omega}. 
\] 
By taking $t$ sufficiently small, we have that 
\[
\int_{D_{k}} |u_{k}|^{2}_{H^{*},\, \omega}e^{-\xi}dV_{\omega} \leq C_{q, \delta} \int_{D_{k}} |f|^{2}_{H^{*},\, \omega}e^{-\xi}dV_{\omega} \leq C_{q, \delta} \int_{D}|f|^{2}_{H^{*},\, \omega}e^{-\xi}dV_{\omega}.  
\]
The constant $C_{q, \delta}$ does not depend on $k$.  
Hence we may choose a subsequence of $\{u_{k}\}_{k \in \mathbb{N}}$ converging weakly in $L^{(n, q-1)}(D, F^{*}, H^{*}e^{-\xi}, \omega)$ to $u$.  
Then $u$ is the $F$-valued $(n, q)$-form we are looking for.  
\end{proof}

\begin{lemma}\label{lemma:2}
Let $\alpha \in L^{(0, q)}(D, F, He^{\xi}, \omega)$ such that $\overline{\partial} \alpha = 0$.  
Assume that $q \geq 1$ and that $\delta < n-q$.  
Then there exists $\beta \in L^{(0, q-1)}(D, F, He^{\xi}, \omega)$ such that $\overline{\partial} \beta = \alpha$ and that 
\[
\int_{D} |\beta|^{2}_{H,\, \omega}e^{\xi} dV_{\omega} \leq C_{q, \delta} \int_{D} |\alpha|^{2}_{H,\, \omega}e^{\xi}dV_{\omega}.  
\]
\end{lemma}
\begin{proof}
Let $\star_{F}$ be the Hodge-star operator 
$L^{(0, q)}(D, F, He^{\xi}, \omega) \to L^{(n, n-q)}(D, F^{*}, H^{*}e^{-\xi}, \omega)$ as in \cite{Cha-Sha}.  
Let $\overline{\partial}_{F^{*}}^{*}$ be the Hilbert space adjoint to $\overline{\partial}:L^{(n, n-q-1)}(D, F^{*}, H^{*}e^{-\xi}, \omega) \to L^{(n, n-q)}(D, F^{*}, H^{*}e^{-\xi}, \omega)$.  
We note that the formal adjoint of $\overline{\partial}$ is equal to $-\star_{F^{*}} \overline{\partial} \, \star_{F}$.  
Since $\omega$ is complete, we have that $\star_{F} \alpha \in L^{(n, n-q)}(D, F^{*}, H^{*}e^{-\xi}, \omega)$ is contained in the domain of $\overline{\partial}_{F^{*}}^{*}$ and $\overline{\partial}_{F^{*}}^{*} \star_{F} \alpha = 0$.  
Lemma~\ref{lemma:1} shows that there exists $u \in L^{(n, n-q+1)}(D, F^{*}, H^{*}e^{-\xi}, \omega)$ such that $\overline{\partial}_{F^{*}}^{*}u = \star_{F}\alpha$ and that $\|u\|^{2}_{H^{*}e^{-\xi},\, \omega} \leq C_{q, \delta} \|\star_{F} \alpha\|^{2}_{H^{*}e^{-\xi},\, \omega} = C_{q, \delta} \|\alpha\|^{2}_{He^{\xi},\, \omega}$.  
By the completeness of $\omega$ again, we have that $\beta = \star_{F^{*}} u \in L^{(0, q-1)}(D, F, He^{\xi}, \omega)$ is contained in the domain of $\overline{\partial}$ and $\overline{\partial} \beta = \alpha$.  
We have that $\|\beta \|^{2}_{He^{\xi},\, \omega} \leq C_{q, \delta}\|\alpha\|^{2}_{He^{\xi},\, \omega}$ and this completes the proof.  
\end{proof}

We denote by $\Omega^{(0, q)}(D, F)$ (resp. $\Omega^{(0, q)}(\overline{D}, F)$) the space of $F$-valued smooth $(0, q)$-forms on $D$ (resp. on a neighborhood of $\overline{D}$).  
Let $\partial ' D = \{z \in \partial D\, |\, \varphi (z) = 0\}$.  
 
\begin{proposition}\label{proposition:1}
Let $1 \leq q \leq n-2$.  
Assume that $\varphi$ is pluriharmonic on a neighborhood of $\partial ' D$ and $d \varphi \neq 0$ on $\partial' D$.  
Let $\alpha \in \Omega^{(0, q)}(\overline{D}, F)$ such that $\overline{\partial} \alpha = 0$ in $D$ and 
that $\mathrm{supp}\, \alpha \cap D \subset \{z \in D\, |\, \rho = 0\}$.  
Then there exists $\beta \in \Omega^{(0, q-1)}(D, F)$ such that $\overline{\partial} \beta = \alpha$ and that $\mathrm{supp}\, \beta \subset \{z \in D \,|\, \rho \leq 1\}$.  
\end{proposition}
\begin{proof}
Let $1 < \delta < 2$.  
If $a \in \{z \in \partial D \, |\, \eta(z) \neq 0\}$, then $a \in \partial' D$.  
There exists a small neighborhood $U \subset X$ of $a$ such that $\varphi$ is pluriharmonic and 
$i \partial \overline{\partial} \phi = i \frac{\partial \varphi \wedge \overline{\partial}\varphi}{\varphi^{2}}$ on $U \cap D$.   
Since $\eta(a) \neq 0$, we may assume that $i \partial \overline{\partial} \rho$ and $\kappa \circ \rho$ are bounded.  
Then 
$e^{\xi}dV_{\omega} \leq  C |\varphi|^{\delta-2}(i \partial \overline{\partial} \psi)^{n}$ on $U \cap D$.  
Hence 
\[
\int_{U \cap D} |\alpha|^{2}_{H,\, \omega}e^{\xi}dV_{\omega} \leq C \int_{U \cap D} |\alpha|^{2}_{H, i \partial \overline{\partial} \psi} |\varphi|^{\delta -2} (i \partial \overline{\partial}\psi)^{n} < + \infty.  
\]
If $b \in \{z \in \partial D \, |\, \eta(z) = 0\}$, there exists a small open neighborhood $U \subset X$ of $b$ such that 
$U \cap \mathrm{supp}\, \alpha = \emptyset$.  
Hence we have that $\alpha \in L^{(0, q)}(D, F, He^{\xi}, \omega)$.  

Let $\kappa_{j} \in C^{\infty}(\mathbb{R})$ ($j = 1, 2, \cdots$) be functions which satisfies the following conditions: 
\begin{itemize}
\item[(i)]
$\kappa_{j}(t) \leq \kappa_{j+1}(t)$ for any $j \in \mathbb{N}$ and $t \geq 0$, 
\item[(ii)]
$\kappa_{j}'(t) \geq 1$ and $\kappa_{j}''(t) \geq 0$ for any $j \in \mathbb{N}$ and $t \geq 0$, 
\item[(iii)]
there exists a positive constant $C$ such that 
$\kappa_{j}(t) \leq C$ for any $j \in \mathbb{N}$ and $0 \leq t \leq 1/2$, 
\item[(iv)]
$\lim_{j \to \infty} \kappa_{j}(t) = +\infty$ for any $t \geq 1$.  
\end{itemize}
Let $\xi_{j} = N\psi + \kappa_{j}\circ \rho - \delta \phi$.  
We have that $\alpha \in L^{(0, q)}(D, F, He^{\xi_{j}}, \omega)$.  
Since $q \leq n-2$, there exist $\beta_{j} \in L^{(0, q-1)}(D, F, He^{\xi_{j}}, \omega)$ such that $\overline{\partial} \beta_{j} = \alpha$ and $\|\beta_{j}\|^{2}_{He^{\xi_{j}},\, \omega} \leq C_{q, \delta}\|\alpha\|^{2}_{He^{\xi_{j}}, \omega}$ by Lemma~\ref{lemma:2}.  
Because $\kappa_{j}\circ \rho \leq C$ on $\mathrm{supp}\, \alpha$, we have that $\|\alpha\|^{2}_{He^{\xi_{j}},\, \omega}$ does not depend on $j$.  
Take a weakly convergent subsequence $\{\beta_{j_{\nu}}\}_{\nu \in \mathbb{N}}$ in $L^{(0, q-1)}(D, F, He^{N\psi - \delta \phi}, \omega)$ 
and there exists the weak limit $\beta \in L^{(0, q-1)}(D, F, He^{N\psi - \delta \phi} ,\omega)$ 
such that $\overline{\partial}\beta = \alpha$ and $\mathrm{supp}\, \beta \subset \{z \in D\, |\, \rho(z) \leq 1\}$.  
The regularity of $\beta$ will be discussed in Section~\ref{section:4}.  
\end{proof}

\section{Estimate of the Sobolev norm}\label{section:4}
It is enough to consider the case $q \geq 2$.  
In the proof of Proposition~\ref{proposition:1}, we take $\kappa_{j} \in C^{\infty}(\mathbb{R})$ ($j \in \mathbb{N}$) which satisfy four conditions.  
We add the following condition to $\{\kappa_{j}\}_{j \in \mathbb{N}}$:  
\begin{itemize} 
\item[(v)]
For any non-negative integer $k$, there exists positive constant $C_{k}$ which does not depend on $j$ and satisfies 
\[
\frac{\kappa_{j}^{(k)}(t)}{(\kappa_{j}(t))^{k+1}} < C_{k}
\]
for any $j \in \mathbb{N}$ and $t \geq 0$.  
\end{itemize}
\begin{lemma}\label{lemmma:3}
There exist functions $\kappa_{j} \in C^{\infty}(\mathbb{R})$ ($j \in \mathbb{N}$) which satisfy the above five conditions.  
\end{lemma}
\begin{proof}
We define $\kappa_{j}(t) = \sum_{l=0}^{j}t^{l}$. 
It is easy to see that $\{\kappa_{j}\}_{j \in \mathbb{N}}$ satisfies the conditions (i), (ii), (iii), (iv).  
For $t \geq 0$, we have that 
\[
(\kappa_{j}(t))^{k+1} \geq \sum_{l = 0}^{j} \binom{l+k}{l}t^{l} = \frac{1}{k!} \sum_{l = 0}^{j} \frac{(l+k)!}{l!} t^{l} \geq \frac{1}{k!} \kappa_{j}^{(k)}(t).  
\]
Hence $\{\kappa_{j}\}_{j \in \mathbb{N}}$ satisfies (v).  
\end{proof}
Assume $\{\kappa_{j}\}_{j \in \mathbb{N}}$ satisfies the above five conditions.   
Take
 $\{\beta_{j}\}_{j \in \mathbb{N}}$ and $\beta$ as in the proof of Proposition~\ref{proposition:1}.  
We may assume that $\beta_{j}$ is orthogonal to the kernel of $\overline{\partial}: L^{(0, q-1)}(D, F, He^{\xi_{j}}, \omega) \to L^{(0, q)}(D, F, He^{\xi_{j}}, \omega)$ and $\beta_{j}$ is smooth (cf.\,\cite{Hor}).  
Let $a \in D$ and let $\chi \in C^{\infty}(D)$ be a non-negative function such that $\chi = 1$ on a neighborhood of $a$ and $\mathrm{supp}\, \chi$ is sufficiently small.  
Denote by $H' = He^{N\psi - \delta \phi}$ the hermitian metric of $F$.  
Then $H'e^{\kappa_{j}\circ \rho} = He^{\xi_{j}}$.  
We may assume that $\mathrm{supp}\, \chi$ is contained in a complex chart $U$ and that $F$ is trivialized there.  
Let  $(x_{1}, \ldots, x_{2n})$ be a local coordinates on $U$.  
Let $K = (k_{1}, k_{2}, \ldots, k_{2n})$ be the multi-index and let $|K| = k_{1} + \cdots + k_{2n}$.  
Define $D^{K}:\Omega^{(0, q-1)}(U, F) \to \Omega^{(0, q-1)}(U, F)$ by $D^{K} = \left(\frac{\partial}{\partial x_{1}}\right)^{k_{1}} \cdots \left(\frac{\partial}{\partial x_{2n}}\right)^{k_{2n}}$.  
\begin{lemma}\label{lemma:4}
Let $k$ be a non-negative integer.  
Then
\[
\sum_{|K| = k}\int_{U} |D^{K} (\chi \beta_{j})|^{2}_{H',\, \omega}e^{\kappa_{j}\circ \rho/2^{k}} dV_{\omega} < C_{k} 
\]
for any $j \in \mathbb{N}$ where $C_{k}$ is a constant which does not depend on $j$.  
\end{lemma}
Let $W^{k}_{(0, q-1)}(U, F)$ be the Sobolev space of $F$-valued $(0, q-1)$-forms whose derivatives up to order $k$ are in $L^{(0, q-1)}(U, F, H', \omega)$.  
Lemma~\ref{lemma:4} implies that the subset $\{\chi \beta_{j}\}_{j \in \mathbb{N}}$ is bounded in $W^{k}_{(0, q-1)}(U, F)$.  
By taking weakly convergent subsequence of $\{ \chi \beta_{j_{\nu}}\}_{\nu}$, we have that $\beta \in W^{k}_{(0, q-1)}(U', F)$ where $U'$ is a sufficiently small open neighborhood of $a$ contained in $U$.  
By the Sobolev lemma (cf.\,\cite{Dem}), it follows that $\beta$ is smooth in $D$.  

The proof of Lemma~\ref{lemma:4} proceeds by induction on $k$.  
We have seen that $\|\beta_{j}\|^{2}_{H'e^{\kappa_{j}\circ \rho},\, \omega}$ does not depend on $j$, and the case $k = 0$ holds.  
Let $\nabla_{j}$ be the hermitian connection on $F$ which is compatible with $H'e^{\kappa_{j}\circ \rho}$.  
By the trivialization of $F$ in $U$, 
we have that $\nabla_{j} = d + \Gamma_{H'} + \partial \kappa_{j} \circ \rho \, \mathrm{Id}_{F}$ where $\Gamma_{H'}$ is the connection form defined by $H'$.  
Let $\ast$ be the Hodge-star operator from $F$-valued $(p, q)$-form to $F$-valued $(n-q, n-p)$-form.  
Define $\vartheta = \ast \partial *: \Omega^{(0, q-1)}(U, F) \to \Omega^{(0, q-2)}(U, F)$.  
Then $-\vartheta$ is the formal adjoint of $\overline{\partial}: L^{(0, q-1)}(U, F, H_{Id}, \omega) \to L^{(0, q)}(U, F, H_{Id}, \omega)$.  
Here $H_{Id}$ is the flat metric of the trivial vector bundle $F$ over $U$.  
\begin{lemma}\label{lemma:5}
Under the hypothesis that Lemma~\ref{lemma:4} holds for $\leq k$,
we have that
\[
\sum_{|K| = k}\int_{U} |\vartheta D^{K}(\chi \beta_{j})|^{2}_{H',\, \omega}e^{\kappa_{j}\circ \rho/2^{k+1}} dV_{\omega} < C_{k+1}
\]
for any $j \in \mathbb{N}$ where $C_{k+1}$ is a constant which does not depend on $j$.  
\end{lemma}
\begin{proof}
We denote by $\overline{\partial}_{j}^{*}$ the Hilbert space adjoint to 
$\overline{\partial}: L^{(0, q-2)}(D, F, H'e^{\kappa_{j}\circ \rho}, \omega) \to L^{(0, q-1)}(D, F, H'e^{\kappa_{j}\circ \rho}, \omega)$.  
The formal adjoint of $\overline{\partial}$ is given by $-\ast (\partial + \Gamma_{H'} + \partial \kappa_{j}\circ \rho \, \mathrm{Id}) \ast$.  
Since $\beta_{j}$ is orthogonal to the kernel of $\overline{\partial}$ in $L^{(0, q-1)}(D, F, H'e^{\kappa_{j}\circ \rho}, \omega)$, we have that $\overline{\partial}_{j}^{*} \beta_{j} = 0$.  
Then $\ast \partial \ast \beta_{j} = - \ast \Gamma_{H'} \ast \beta_{j} - \ast \partial \kappa_{j} \circ \rho \wedge \beta_{j}$.  
Hence 
\[
\ast \partial \ast (\chi \beta_{j}) = \ast (\partial \chi \wedge \ast \beta_{j}) - \ast \Gamma_{H'} \ast (\chi \beta_{j}) - \partial \kappa_{j} \circ \rho \wedge (\chi \beta_{j}).  
\]
Let $K = (k_{1}, \ldots, k_{2n})$ such that $k = k_{1} + \cdots + k_{2n}$.  
It follows that the order of the differential operator $\ast D^{K} - D^{K} \ast$ is $k-1$.  
Then 
\begin{align*}
\vartheta D^{K} (\chi \beta_{j}) = \ast \partial \ast D^{K} (\chi \beta_{j}) = D^{K} \ast \partial \ast (\chi \beta_{j}) + L \beta_{j}, 
\end{align*}
where $L$ is the differential operator of order $k$.  
The term $D^{K}\ast \partial \ast (\chi \beta_{j})$ is written as 
\[
\sum_{k' \leq k,\, l \leq k+1} \sum_{|K'| = k'} t_{K', l}\kappa_{j}^{(l)} \circ \rho D^{K'} \beta_{j}   
\]
where $t_{K', l}$ is the function which does not depend on $j$.  
It follows that 
\begin{align*}
& \int_{U} |\vartheta D^{K}(\chi \beta_{j})|^{2}_{H',\,\omega}e^{\kappa_{j}\circ \rho/2^{k+1}}dV_{\omega} \\
\leq 
& \sum_{k' \leq k,\, l \leq k+1}\sum_{|K'| = k'} \int_{U} (|t_{K', l}\kappa_{j}^{(l)}\circ \rho|^{2} |D^{K'}\beta_{j}|^{2}_{H',\, \omega} + |L \beta_{j}|^{2}_{H',\, \omega}) e^{\kappa_{j}\circ \rho/2^{k+1}} dV_{\omega} \\
\leq 
& C \sum_{k' \leq k,\, l \leq k+1}\sum_{|K'| = k'} \int_{U} \frac{|\kappa_{j}^{(l)}\circ \rho|^{2}}{(\kappa_{j}\circ \rho)^{2(l+1)}} |D^{K'}\beta_{j}|_{H',\, \omega} e^{\kappa_{j}\circ \rho/2^{k}} (\kappa_{j} \circ \rho)^{2(l+1)} e^{- \kappa_{j}\circ \rho/2^{k+1}}  dV_{\omega}, 
\end{align*}
where $C$ does not depend on $j$.  
Since $\sup_{r \geq 0}r^{2(l+1)} e^{-r/2^{k+1}} < + \infty$, 
the condition (v) of $\{\kappa_{j}\}_{j \in \mathbb{N}}$ shows that 
the last term of the above inequality bounded from above by the constant which does not depend on $j$.  
\end{proof}
A standard calculus shows that 
Lemma~\ref{lemma:5} implies Lemma~\ref{lemma:4} (cf.\,Chapter~5 of \cite{Hor}).  
\begin{proof}[Proof of Lemma~\ref{lemma:4}]
We may assume that there exists an orthonormal frame $(\theta_{1}, \ldots, \theta_{n})$ of $T^{*}_{X}$ on $U$.  
For any $f \in C^{\infty}(U)$, 
we define $\partial f / \partial \theta_{t}$ and $\partial f /\partial \bar{\theta}_{t}$ by $df = \sum_{t = 1}^{n} \partial f /\partial \theta_{t}\, \theta_{t} + \partial f /\partial \bar{\theta}_{t}\, \bar{\theta}_{t}$.  
Let $I = (i_{1}, \ldots, i_{q-1})$ be a multi-index with $i_{1} < \cdots < i_{q-1}$ and $\bar{\theta}_{I} = \bar{\theta}_{i_{1}} \wedge \cdots \wedge \bar{\theta}_{i_{q-1}}$.  
Let $g = \sum_{|I| = q-1} g_{I}\bar{\theta}_{I}$ be a smooth $(0, q-1)$-form on $U$ with compact support.  
Define $Ag = \sum_{|I| = q-1}\sum_{l = 1}^{n} \frac{\partial g_{I}}{\partial \bar{\theta_{l}}} \bar{\theta}_{l}\wedge \bar{\theta}_{I}$ and 
$Bg = \sum_{|J| = q-2} \sum_{l} \frac{\partial g_{lJ}}{\partial \theta_{l}} \bar{\theta}_{J}$.  
Then $\overline{\partial} g - Ag$ and $-\vartheta g - Bg$ have no term where $g_{I}$ is differentiated.    
Hence we have that 
\begin{equation}\label{equation:1}
\int_{U}(|Ag|^{2}_{\omega} + |Bg|_{\omega}^{2})e^{\kappa_{j}\circ \rho/2^{k+1}} dV_{\omega} \leq \int_{U}(2|\overline{\partial}g|_{\omega}^{2} + 2|\vartheta g|_{\omega}^{2} + C|g|^{2}_{\omega}) e^{\kappa_{j}\circ \rho/2^{k+1}} dV_{\omega}.   
\end{equation}
The left-hand side of the above inequality is equal to 
\begin{align*}
& \sum_{|I| = q-1}\sum_{l = 1}^{n} \int_{U}\big|\frac{\partial g_{I}}{\partial \bar{\theta}_{l}} \big|^{2}e^{\kappa_{j}\circ \rho/2^{k+1}}dV_{\omega}
- \sum_{|J| = q-2} \sum_{l, t} \int_{U}\left(\frac{\partial g_{lJ}}{\partial \bar{\theta}_{t}}\frac{\partial \bar{g}_{tJ}}{\partial \theta_{l}} - \frac{\partial g_{lJ}}{\partial \theta_{l}} \frac{\partial \bar{g}_{tJ}}{\partial \bar{\theta}_{t}} \right)e^{\kappa_{j}\circ \rho/2^{k+1}} dV_{\omega}.  
\end{align*}
By integrating by parts, the second integral of the above is equal to 
\[
\int_{U} \left(\frac{\partial^{2}g_{lJ}}{\partial \bar{\theta}_{t} \partial \theta_{l}} - \frac{\partial^{2} g_{lJ}}{\partial \theta_{l} \partial \bar{\theta}_{t}}\right)\bar{g}_{tJ}e^{\kappa_{j} \circ \rho/2^{k+1}}dV_{\omega} + \int_{U}R e^{\kappa_{j} \circ \rho/2^{k+1}}dV_{\omega}.  
\]
Here $R$ is written as the sum of 
$s_{1}\, g_{I_{1}} \kappa_{j}^{(i)}\circ \rho \, \partial g_{I_{2}}/\partial \bar{\theta}_{t}$ and $s_{2}\, g_{I_{1}} g_{I_{2}} \kappa^{(i)}\circ \rho$ ($i = 0, 1, 2$) where $s_{1}$ and $s_{2}$ are the smooth functions which depend on neither $g$ nor $j$.   
The order of the differential operator 
$
\partial^{2} / \partial \theta_{l} \partial \bar{\theta}_{t} - \partial^{2} / \partial \bar{\theta}_{t} \partial \theta_{l}
$
is one.  
The condition (v) of $\{\kappa_{j}\}_{j \in \mathbb{N}}$ shows that 
\begin{align*}
& \int_{U} \big| g_{I_{1}}\frac{\partial g_{I_{2}}}{\partial \bar{\theta}_{t}} \big| \kappa_{j}^{(i)}\circ \rho \, e^{\kappa_{j}\circ \rho /2^{k+1}}dV_{\omega} \\
\leq & 
\frac{1}{\varepsilon}\int_{U}|g_{I_{1}} \kappa_{j}^{(i)}\circ \rho|^{2} \, e^{\kappa_{j}\circ \rho/2^{k+1}} dV_{\omega} 
+ \varepsilon \int_{U} \big| \frac{\partial g_{I_{2}}}{\partial \bar{\theta}_{t}}\big|^{2} e^{\kappa_{j}\circ \rho/2^{k+1}}dV_{\omega} \\
\leq & 
C_{\varepsilon}\int_{U}|g_{I_{1}}|^{2}\ e^{\kappa_{j}\circ \rho/2^{k}} dV_{\omega} 
+ \varepsilon \int_{U} \big| \frac{\partial g_{I_{2}}}{\partial \bar{\theta}_{t}}\big|^{2} e^{\kappa_{j}\circ \rho/2^{k+1}}dV_{\omega} 
\end{align*}
for any $\varepsilon > 0$.  
Hence the left-hand side of (\ref{equation:1}) is bounded below by 
\[
\frac{1}{2}\sum_{|I| = q-1}\sum_{l = 1}^{n} \int_{U}\big|\frac{\partial g_{I}}{\partial \bar{\theta}_{l}} \big|^{2}e^{\kappa_{j}\circ \rho/2^{k+1}}dV_{\omega} - C\int_{U} |g|^{2}_{\omega} e^{\kappa_{j}\circ \rho/2^{k}}dV_{\omega},  
\]
where $C$ does not depend on $j$.  
Because the order of the differential operator $\partial^{2}/ \partial \theta_{l}\partial\bar{\theta}_{l} - \partial^{2}/ \partial\bar{\theta}_{l} \partial \theta_{l}$ is one, we have that 
\begin{align*}
& \int_{U} \big| \frac{\partial g_{I}}{\partial \theta_{l}} \big|^{2} e^{\kappa_{j} \circ \rho/2^{k+1}} dV_{\omega} - \int_{U} \big| \frac{\partial g_{I}}{\partial \bar{\theta}_{l}} \big|^{2} e^{\kappa_{j} \circ \rho/2^{k+1}} dV_{\omega} \\
\leq & 
\varepsilon \int_{U} \left( \big| \frac{\partial g_{I}}{\partial \theta_{l}} \big|^{2} + \big| \frac{\partial g_{I}}{\partial \bar{\theta}_{l}} \big|^{2} \right) e^{\kappa_{j}\circ \rho/2^{k+1}} dV_{\omega} + C_{\varepsilon} \int_{U} |g_{I}|^{2} e^{\kappa_{j}\circ \rho/2^{k}} dV_{\omega} 
\end{align*}
for any $\varepsilon > 0$.  
Since $F$ is trivialized in $U$ and $H'$ is equivalent to the flat metric $H_{\mathrm{Id}}$, 
we can replace $g$ by $D^{K}(\chi \beta_{j})$ in the above calculus, and we obtain 
\begin{align*}
& \sum_{l = 1}^{n}\int_{U} \left( \big|\frac{\partial D^{K} (\chi \beta_{j})}{\partial \theta_{l}} \big|^{2}_{H',\, \omega} + \big| \frac{\partial D^{K}(\chi \beta_{j})}{\partial \bar{\theta}_{l}}\big|^{2}_{H',\,\omega}\right) e^{\kappa_{j}\circ \rho /2^{k+1}} dV_{\omega} \\
\leq & C \int_{U} (|D^{K}\overline{\partial} (\chi \beta_{j})|^{2}_{H', \,\omega} + |\vartheta D^{K}(\chi \beta_{j})|_{H', \,\omega}^{2})e^{\kappa_{j}\circ \rho/2^{k+1}}dV_{\omega} + C \int_{U} |D^{K}(\chi \beta_{j})|_{H',\, \omega}^{2}e^{\kappa_{j}\circ \rho/2^{k}} dV_{\omega}.  
\end{align*}
Note that $\overline{\partial} (\chi \beta_{j}) = \chi \alpha + \overline{\partial}\chi \wedge \beta_{j}$.  
The induction hypothesis and Lemma~\ref{lemma:5} show that the last term of the above inequality does not depend on $j$.  
\end{proof}

\section{Extension of closed $F$-valued forms in manifolds}\label{section:5}
In this section, we extend closed $F$-valued forms as in \cite{Tib18}.  

Let $X$ be a Stein manifold and let $\varphi$ be a non-constant plurisubharmonic function on $X$ such  that for every $r < \sup_{X}\varphi$ the sublevel set $\{z \in X\, |\, \varphi(z) < r\}$ is compact. 
Let $F$ be a holomorphic vector bundle over $X$.  
Let $\psi \in C^{\infty}(X)$ be an exhaustive strictly plurisubharmonic function.  
We define $D_{\varphi}(r) = \{z\in X\,|\, \varphi(z)<r\}$.

\begin{lemma}\label{lemma:6}
Let $U \subset X$ be an open neighborhood of $\mathrm{supp}\, i \partial \overline{\partial} \varphi$ 
and let $u \in \Omega^{(0, q-1)}(U, F)$ ($1 \leq q \leq n-2$) such that $\overline{\partial}u = 0$.  
Let $r < \sup \varphi$ such that $d\varphi \neq 0$ on $\partial D_{\varphi}(r) \setminus \mathrm{supp}\, i\partial \overline{\partial}\varphi$ (Note that $\varphi$ is smooth on $X \setminus \mathrm{supp}\, i\partial \overline{\partial} \varphi$).  
Then there exists $v \in \Omega^{(0, q-1)}(D_{\varphi}(r), F)$ such that $u = v$ on a neighborhood of $\mathrm{supp}\, i\partial \overline{\partial} \varphi \cap D_{\varphi}(r)$.  
\end{lemma}

\begin{proof}
We take $\chi \in C^{\infty}(X)$ such that $0 \leq \chi \leq 1$, $\chi = 1$ on a neighborhood of $\mathrm{supp}\, i\partial \overline{\partial} \varphi$ and that $\mathrm{supp}\, \chi \subset U$.  
Then $\alpha := \overline{\partial} (\chi u)$ is $\overline{\partial}$-closed.  
By Lemma~\ref{lemma:7} below, there exists a smooth plurisubharmonic function $\varphi_{\varepsilon}$ on a neighborhood of $\overline{D_{\varphi}(r)}$ such that $\varphi_{\varepsilon} \geq \varphi$ on $\overline{D_{\varphi}(r)}$, $\varphi_{\varepsilon} > \varphi$ on $\mathrm{supp}\, i\partial \overline{\partial} \varphi \cap \overline{D_{\varphi}(r)}$, and that $\varphi_{\varepsilon} = \varphi$ on $\mathrm{supp}\, \alpha$.  

Let $\eta = \varphi_{\varepsilon}- \varphi$.  
It follows that $\eta$ is plurisubharmonic on $D_{\varphi}(r) \setminus \mathrm{supp}\, i \partial \overline{\partial} \varphi$ and that $\eta = 0$ on $\mathrm{supp}\, \alpha$.  
Since $\eta$ is lower semi-continuous on $\overline{D_{\varphi}(r)}$, 
there exists $c > 0$ such that $\eta > c$ on $\overline{D_{\varphi}(r)} \cap \mathrm{supp}\, i \partial \overline{\partial} \varphi$.  
Let $D' = \{z \in D \,|\, \varphi(z) -r< 0, \eta(z) - c/2 < 0 \}$.  
Note $\eta \in C^{\infty}(\overline{D'})$ since $\varphi \in C^{\infty}(D \setminus \mathrm{supp}\, i\partial \overline{\partial} \varphi)$.
By Proposition~\ref{proposition:1}, there exist $c' < \frac{c}{2}$ and $\beta \in \Omega^{(0, q-1)}(D', F)$ such that $\overline{\partial} \beta = \alpha$ and that $\mathrm{supp}\, \beta \subset \{z \in D'\, |\, \eta(z) \leq c' \}$.  
(If $D'$ is a disjoint union of bounded pseudoconvex domains, we apply Proposition~\ref{proposition:1} to each component.)  
We extend $\beta$ by $0$ on $D_{\varphi}(r) \setminus D'$, and 
we consider $\beta$ as an element of $\Omega^{(0, q-1)}(D_{\varphi}(r), F)$.  
We have that $\overline{\partial}\beta = \alpha$ on $D_{\varphi}(r)$.  
Then $v = \chi u - \beta$ is the form we are looking for.  
\end{proof}

\begin{lemma}\label{lemma:7}
Let $V \subset X$ be an open neighborhood of $\mathrm{supp}\, i\partial \overline{\partial} \varphi$.  
Let $t < \sup \varphi$.  
There exists a smooth plurisubharmonic function $\varphi_{\varepsilon}$ on $D_{\varphi}(t)$ such that $\varphi_{\varepsilon} \geq \varphi$ on $D_{\varphi}(t)$, $\varphi_{\varepsilon} > \varphi$ on $\mathrm{supp}\, i\partial \overline{\partial} \varphi$, and that $\varphi_{\varepsilon} = \varphi$ on $D_{\varphi}(t) \setminus V$.  
\end{lemma} 
\begin{proof}
Since $X$ is Stein, we may assume that $X$ is a submanifold of $\mathbb{C}^{m}$.  
By the theorem of Docquier and Grauert, there exists an open neighborhood $W \subset \mathbb{C}^{m}$ of $X$ and a holomorphic retraction $\mu: W \to X$ (cf.\,Chapter VIII of \cite{Gun-Ros}).  
Let $h: \mathbb{C}^{m} \to \mathbb{R}^{+}$ be a smooth function depending only on $|z|$ ($z \in \mathbb{C}^{m}$) whose support is contained in the unit ball and whose integral is equal to one.  
Define $h_{\varepsilon}(z) = (1/\varepsilon^{2m})h(z/\varepsilon)$ for $\varepsilon>0$.  
Let $W' \subset W$ be a relatively compact open neighborhood of $D_{\varphi}(t)$.  
If $\varepsilon > 0$ is sufficiently small, we can define $(\varphi \circ \mu)_{\varepsilon} = (\varphi \circ \mu) * h_{\varepsilon}$ on $W'$.  
Put $\varphi_{\varepsilon} = (\varphi \circ \mu)_{\varepsilon}$ on $D_{\varphi}(t)$.  
Then we have $\varphi_{\varepsilon} \geq \varphi$ on $D_{\varphi}(t)$ and $\varphi_{\varepsilon} > \varphi$ on $\mathrm{supp}\, i\partial \overline{\partial}\varphi$.  
Since $\mu^{-1}(\mathrm{supp}\, i\partial \overline{\partial}\varphi) \cap (W' \setminus \mu^{-1}(V)) = \emptyset$, 
there exists small $\varepsilon > 0$ such that the ball of radius $\varepsilon$ centered at every point of $W' \setminus \mu^{-1}(V)$ does not intersect $\mu^{-1}(\mathrm{supp}\, i\partial \overline{\partial} \varphi)$.  
Then $(\varphi \circ \mu)_{\varepsilon} = \varphi \circ \mu$ on $W' \setminus \mu^{-1}(V)$.  
\end{proof}

\begin{lemma}\label{lemma:a}
Let $\varphi$ be a plurisubharmonic function on $X$ as in Theorem~\ref{theorem:1}.  
Let  $r < s < t < \sup_{X}\varphi$ such that $D_{\varphi}(r) \subset \subset D_{\varphi}(s) \subset \subset D_{\varphi}(t)$.    
Then there exists a plurisubharmonic function $\phi$ on $X$ which satisfies the following conditions:  
\begin{itemize}
\item[(a)]
$\varphi = \phi$ on $\{z \in D\,|\, \varphi(z) \geq t\}$.  
\item[(b)]
$\mathrm{supp}\, i\partial \overline{\partial} \varphi \cup \overline{D_{\varphi}(r)} \subset \mathrm{supp}\, i\partial \overline{\partial} \phi$.  
\end{itemize}
\end{lemma}
\begin{proof}
Let $\psi$ be an exhaustive strongly plurisubharmonic function on $X$.   
We may assume that $\psi > 0$ on $X$.  
Let $M = \sup_{\overline{D_{\varphi}(t)}} \psi$ and 
let $\tau(z) = r + \frac{(s-r)\psi}{M +1}$.  
Then $r < \tau < s$ on the closure of $D_{\varphi}(t)$.  
We define 
\begin{equation*}
\phi(z) = \left\{  
\begin{aligned}
&\varphi(z) & \text{if $z \in \{x \in X\, |\, \varphi(x) \geq t$\}} \\
&\max\{\varphi(z), \tau(z)\} & \text{if $z \in D_{\varphi}(t)$}\qquad \qquad \quad \,\,\,
\end{aligned}
\right.
\end{equation*}
Then $\phi$ is a plurisubharmonic function on $X$.   
It is easy to see that $\phi$ satisfies the conditions of the lemma.  
\end{proof}
\begin{lemma}\label{lemma:b}
Let $X$ and $\varphi$ be as in Theorem~\ref{theorem:1}.  
Let $U \subset X$ be an open neighborhood of $\mathrm{supp}\, i \partial \overline{\partial} \varphi$ and let $u \in \Omega^{(0, q-1)}(U, F)$ ($1 \leq q \leq n-2$) such that $\overline{\partial} u = 0$.  
Let $s_{1} < s_{2} < \cdots < \sup_{X}\varphi$ such that $\lim_{j \to \infty}s_{j} = \sup_{D} \varphi$ and that 
\[
D_{\varphi}(s_{1}) \subset \subset D_{\varphi}(s_{2}) \subset \subset \cdots \subset \subset X.
\]  
Then there exist $V_{j} \subset X$ and $v_{j} \in \Omega^{(0, q-1)}(V_{j}, F)$ ($j = 3, 4, \ldots$) which satisfy the following:  
\begin{itemize}
\item[(i)]
$V_{j}$ is an open neighborhood of $\mathrm{supp}\, i \partial \overline{\partial} \varphi \cap \overline{D_{\varphi}(s_{j})}$.  
\item[(ii)]
$\overline{\partial}v_{j} = 0$ on $V_{j}$.  
\item[(iii)]
$v_{j + 1} = v_{j}$ on $D_{\varphi}(s_{j-2})$.   
\item[(iv)]
$v_{j} = u$ on a neighborhood of $\mathrm{supp}\, i \partial \overline{\partial} \varphi$.  
\end{itemize}
\end{lemma}
\begin{proof}
By Lemma~\ref{lemma:6}, there exists $v \in \Omega^{(0, q-1)}(D_{\varphi}(s_{5}), F)$ such that $\overline{\partial}v = 0$ and that $v = u$ on a neighborhood of $\mathrm{supp}\, i \partial \overline{\partial} \varphi \cap D_{\varphi}(s_{5})$.  
We may assume that $U$ is sufficiently small. 
Then $v$ on $D_{\varphi}(s_{4})$ and $u$ on $U$ can be glued together to give the form $v_{3}$ on $V_{3} := U \cup D_{\varphi}(s_{4})$.  

Assume that there exist $V_{j}$ and $v_{j}$ which satisfy the condition of the lemma ($j \geq 3$).  
By Lemma~\ref{lemma:a}, there exists plurisubharmonic function $\phi$ on $X$ such that $\varphi = \phi$ on $\{z \in D \, |\, \varphi(z) \geq s_{j}\}$ and that $\mathrm{supp}\, i\partial \overline{\partial} \varphi \cup \overline{D_{\varphi}(s_{j-2})} \subset \mathrm{supp}\, i \partial \overline{\partial} \phi$.  
Then $v_{j}$ is defined on a neighborhood of $\mathrm{supp}\, i\partial \overline{\partial} \phi$.  
By Lemma~\ref{lemma:6}, there exists $\tilde{v} \in \Omega^{(0, q-1)}(D_{\varphi}(s_{j+3}), F)$ such that $\overline{\partial}\tilde{v} = 0$ and that $\tilde{v} = v_{j}$ on an open neighborhood of $D_{\varphi}(s_{j+3}) \cap \mathrm{supp}\, i\partial \overline{\partial} \phi$.  
As in the case of $v_{3}$, we can glue $\tilde{v}$ and $v_{j}$ together, and we obtain $v_{j+1}$ and $V_{j+1}$ which satisfy the conditions of the lemma.  
Hence we obtain $V_{k}$ and $v_{k}$ ($k = 3, 4, \ldots$) inductively.  
\end{proof}

Now we prove that the natural map 
\[
H^{0}(X, F) \to \underset{\mathrm{supp}\, i \partial \overline{\partial}\varphi \subset V}{\varinjlim}H^{0}(V, F).  
\]
is an isomorphism and 
\[
\underset{\mathrm{supp}\, i \partial \overline{\partial}\varphi \subset V}{\varinjlim}H^{q}(V, F) = 0.  
\]
for $0< q < n-2$.  
\begin{proof}
Let $U$ be an open neighborhood of $\mathrm{supp}\, i \partial \overline{\partial} \varphi$ in $X$.  
Let $u \in \Omega^{(0, q-1)}(U, F)$ ($1 \leq q \leq n-2$) such that $\overline{\partial} u = 0$.  
Take $\{s_{j}\}$ and $\{v_{j}\}$ as in Lemma~\ref{lemma:5}.  
We can define $v \in \Omega^{(0, q-1)}(X, F)$ by $v(z) = v_{j+2}(z)$ when $z \in D_{\varphi}(s_{j})$.  
Then $\overline{\partial} v = 0$ and $v$ is equal to $u$ on a neighborhood of $\mathrm{supp}\, i \partial \overline{\partial}\varphi$.   
Hence the natural map $H^{q-1}(X, F) \to \underset{\mathrm{supp}\, i\partial \overline{\partial} \varphi \subset V}\varinjlim H^{q-1}(V, F)$ is surjective.  
Since $H^{q-1}(X, F) = 0$ for $q \geq 2$, this completes the proof.  
\end{proof}

Let $X$ be a projective manifold and let $T$ be a closed positive current of type $(1,1)$ on $X$ such that the cohomology class $\{T\}$ belongs to $\mathcal{K}_{NS}$.   
There exist very ample line bundles $L_{1}, \ldots, L_{p}$ and positive numbers $a_{1}, \ldots, a_{p}$ such that $\{T\} = a_{1}c_{1}(L_{1}) + \cdots + a_{p}c_{1}(L_{p})$ where $c_{1}(L_{j})$ is the first Chern class of $L_{j}$.  
Let $\omega_{j}$ be a smooth closed positive form such that $\omega_{j} \in c_{1}(L_{j})$.  
Put $\omega = a_{1}\omega_{1} + \cdots a_{p}\omega_{p}$.  
Then there exists an almost plurisubharmonic function $\varphi$ on $X$ such that $T = \omega + i \partial \overline{\partial} \varphi$ (see Section~14 of \cite{Dem2}).   
Here we say that a function $\varphi$ is almost plurisubharmonic if, for any $x \in X$, there exists a smooth function $\psi$ on a neighborhood of $x$ such that $\varphi + \psi$ is a plurisubharmonic function.  
Then we have that points where $\varphi$ is not continuous belong to $\mathrm{supp}\, T$.  
(We note that this claim needs not hold if $\varphi$ is not almost plurisubharmonic, 
even when $\varphi$ is a locally integrable upper-semicontinuous function.)

First we assume that $\varphi$ is bounded on $X$.  
Take non-zero holomorphic sections $s_{1} \in \Gamma(X, L_{1}), \ldots, s_{p} \in \Gamma(X, L_{p})$.  
Define $s = s_{1} \otimes \cdots \otimes s_{p} \in \Gamma(X, L_{1} \otimes \cdots \otimes L_{p})$.  
\begin{lemma}\label{lemma:8}
Let $U \subset X$ be an open neighborhood of $\mathrm{supp}\, T = \mathrm{supp}\, (\omega + i\partial \overline{\partial} \varphi)$.  
Here $\varphi$ is a bounded almost plurisubharmonic function.  
Let $u \in \Omega^{(0, q-1)}(U, F)$ ($1 \leq q \leq n-2$) such that $\overline{\partial} u = 0$.  
Let $K \subset X \setminus \{z \in X\, |\, s(z) = 0\}$ be a compact set.  
Then there exist an open neighborhood $U_{1} \subset X$ of $K \cup \mathrm{supp}\, (\omega + i \partial \overline{\partial} \varphi)$, 
a $\overline{\partial}$-closed $F$-valued form $u_{1} \in \Omega^{(0, q-1)}(U_{1}, F)$ and a bounded almost plurisubharmonic function $\varphi_{1}$ on $X$ which satisfy the following conditions:  
\begin{itemize}
\item[(a)]
$u = u_{1}$ on an open neighborhood of $\mathrm{supp}\, (\omega + i \partial \overline{\partial} \varphi)$, 
\item[(b)]
$\omega + i\partial \overline{\partial} \varphi_{1} \geq 0$, 
\item[(c)]
$K \cup \mathrm{supp}\,(\omega + i\partial \overline{\partial} \varphi) \subset \mathrm{supp}\, (\omega + i\partial \overline{\partial} \varphi_{1}) \subset U_{1}$.  
\end{itemize}
\end{lemma}
\begin{proof} 
Put $Y = \{z \in X\, |\, s(z) = 0\}$.  
Then $X \setminus Y$ is a Stein manifold.  
Let $\|\cdot\|_{j}$ be a smooth hermitian metric of $L_{j}$ whose Chern curvature is $\omega_{j}$.  
Let $\tilde{\varphi} = \varphi -\sum_{j = 1}^{p}\frac{1}{2\pi}\log \|s_{j}\|^{2a_{j}}$ on $X \setminus Y$.  
Then $\tilde{\varphi}$ is an exhaustive plurisubharmonic function and $\mathrm{supp}\, i \partial \overline{\partial} \tilde{\varphi} = (X \setminus Y) \cap \mathrm{supp}\, (\omega + i\partial \overline{\partial} \varphi)$.  
Let $D_{\tilde{\varphi}}(r) = \{z \in X \setminus Y\, |\, \tilde{\varphi}(z) < r\}$.  
Take $s < t$ such that $K \subset D_{\tilde{\varphi}}(s) \subset \subset D_{\tilde{\varphi}}(t)$.  
There exists $v \in \Omega^{(0, q-1)}(D_{\tilde{\varphi}}(t), F)$ such that $\overline{\partial} v = 0$ and that $u = v$ on an open neighborhood of $D_{\tilde{\varphi}}(s) \cap \mathrm{supp}\, i \partial \overline{\partial} \tilde{\varphi}$ by Lemma~\ref{lemma:6}.  
As in the proof of Lemma~\ref{lemma:b}, 
we can glue $v$ and $u$ together, and we obtain an open neighborhood $U_{1} \subset X$ of $K \cup \mathrm{supp}\, (\omega + i \partial \overline{\partial} \varphi)$ and $u_{1} \in \Omega^{(0, q-1)}(U_{1}, F)$ which satisfy (a).  
By lemma~\ref{lemma:a}, there exists bounded from below, exhaustive plurisubharmonic function $\tilde{\varphi}_{1}$ such that $K \cup \mathrm{supp}\, i \partial\overline{\partial} \tilde{\varphi} \subset \mathrm{supp}\, i\partial \overline{\partial} \tilde{\varphi}_{1} \subset U_{1}$ 
and that $\tilde{\varphi}_{1} = \tilde{\varphi}$ on $X \setminus (Y \cup D_{\tilde{\varphi}}(t))$.  
Put $\varphi_{1} = \tilde{\varphi}_{1} + \sum_{j=1}^{p}\frac{1}{2\pi}\log \|s\|^{2a_{j}}$.  
Since $\varphi_{1} = \varphi$ on a neighborhood of $Y$, we can consider $\varphi_{1}$ as a bounded almost plurisubharmonic function on $X$.  
Then $\varphi_{1}$ is a function we are looking for.  
\end{proof}

Now we prove that the natural map 
\[
H^{q}(X, F) \to \underset{\mathrm{supp}\, (\omega + i\partial \overline{\partial} \varphi) \subset V}{\varinjlim} H^{q}(V, F) 
\]
is an isomorphism for $0 \leq q < n-2$ and is injective for $q = n-2$.  
\begin{proof}[proof of the case where $\varphi$ is bounded]
We first show that $H^{q}(X, F) \to \underset{\mathrm{supp}\, (\omega + i\partial \overline{\partial} \varphi) \subset V}\varinjlim H^{q}(V, F)$ is surjective for $0 \leq q \leq n-3$.  
Let $u \in \Omega^{(0, q)}(U, F)$ be a $\overline{\partial}$-closed $F$-valued form where $U$ is an open neighborhood of $\mathrm{supp}\, (\omega + i\partial \overline{\partial} \varphi)$.  
We can take compact sets $K_{l} \subset X$ and $s_{l, j} \in \Gamma(X, L_{j})$ ($1 \leq l \leq N$, $1 \leq j \leq p$) such that $\bigcup_{l = 1}^{N} K_{l} = X$ and that $K_{l} \cap \bigcup_{j=1}^{p} \{z \in X\, |\, s_{l, j}(z) = 0\} = \emptyset$ for any $l$.  
Put $\varphi_{0} = \varphi$, $u_{0} = u$.  
By using Lemma~\ref{lemma:8} repeatedly, there exist $U_{l} \subset X$, $u_{l} \in \Omega^{(0, q)}(U_{l}, F)$ ($1 \leq l \leq N$) and $\varphi_{l}$ ($1 \leq l \leq N-1$)
such that  
$\bigcup_{k=1}^{l}K_{k} \cup \mathrm{supp}\,(\omega + i\partial\overline{\partial}\varphi) \subset \mathrm{supp}\,(\omega + i\partial \overline{\partial} \varphi_{l}) \subset U_{l}$ and that $u_{l} = u_{l-1}$ on a neighborhood of $\mathrm{supp}\, (\omega + i \partial \overline{\partial} \varphi_{l-1})$.  
Hence $u_{N} \in \Omega^{(0, q)}(X, F)$ is $\overline{\partial}$-closed $F$-valued form such that $u_{N} = u$ on a neighborhood of $\mathrm{supp}\, (\omega + i \partial \overline{\partial} \varphi)$. 
This proves the surjectivity.   

Next, we show that $H^{q}(X, F) \to \underset{\mathrm{supp}\, (\omega + i\partial \overline{\partial} \varphi) \subset V}\varinjlim H^{q}(V, F)$ is injective for $0 \leq q \leq n-2$.  
Let $\alpha \in \Omega^{(0, q)}(X, F)$ such that $\overline{\partial} \alpha = 0$.  
Suppose that there exists an open neighborhood $U \subset X$ of $\mathrm{supp}\, (\omega + i\partial \overline{\partial} \varphi)$ and $v \in \Omega^{(0, q-1)}(U, F)$ such that $\alpha = \overline{\partial} v$ on $U$.  
Let $\chi \in C^{\infty}(X)$ be a function such that $\mathrm{supp}\, \chi \subset U$ and that $\chi = 1$ on a neighborhood of $\mathrm{supp}\, (\omega + i\partial \overline{\partial} \varphi)$.  
Then $\alpha - \overline{\partial} (\chi v)$ is $\overline{\partial}$-closed $F$-valued form which vanishes on a neighborhood of $\mathrm{supp}\, (\omega + i\partial \overline{\partial} \varphi)$.   
Take $K_{1}, \ldots, K_{N}$ as in the proof of the surjectivity.  
As in the proof of Lemma~\ref{lemma:6}, there exists 
$F$-valued $(0, q-1)$-form $v'_{1}$ which is defined in an open neighborhood of $K_{1}$ such that $\overline{\partial} v'_{1} = \alpha - \overline{\partial}(\chi v)$ and that $v'_{1} = 0$ on a neighborhood of $\mathrm{supp}\, (\omega + i\partial \overline{\partial} \varphi)$.  
By the trivial extension, we may assume that $v'_{1}$ is defined on a neighborhood $U_{1}$ of $K_{1}  \cup \mathrm{supp}\, (\omega + i\partial \overline{\partial} \varphi)$.  
Define $v_{1} = \chi v + v'_{1}$ 
on $U_{1}$.  
We have that $\overline{\partial}v_{1} = \alpha$ and that $v_{1} = v$ on a neighborhood of $\mathrm{supp}\, (\omega + i\partial \overline{\partial} \varphi)$.  
As in Lemma~\ref{lemma:8}, there exists a bounded function $\varphi_{1}$ such that $\omega + i\partial \overline{\partial} \varphi_{1} \geq 0$ and that $K_{1} \cup \mathrm{supp}\, (\omega + i\partial \overline{\partial} \varphi) \subset \mathrm{supp}\, (\omega + i\partial \overline{\partial} \varphi_{1}) \subset U_{1}$.  
If we replace $K_{1}, v, U, \varphi$ by $K_{2}, v_{1}, U_{1}, \varphi_{1}$ respectively, we obtain $v_{2}, U_{2}, \varphi_{2}$ which satisfy the suitable conditions.  
By repeating this process, we obtain $v_{N} \in \Omega^{(0, q-1)}(X, F)$ such that $\overline{\partial} v_{N} = \alpha$.  
\end{proof}
\begin{proof}[proof of the case where $\varphi$ is unbounded]
Define $\varphi_{c} = \max\{\varphi, c\}$ for $c \in \mathbb{R}$ and 
$T_{c} = \omega + i \partial \overline{\partial} \varphi_{c}$.  
Then $\varphi_{c}$ is bounded and $T_{c} \geq 0$ since $\omega > 0$.  
It is easy to see that $\mathrm{supp}\, T = \mathrm{supp}\, (\omega + i \partial \overline{\partial} \varphi) \subset \mathrm{supp}\, T_{c}$.  
Let $U$ be any open neighborhood of $\mathrm{supp}\, T$, 
We show that $\mathrm{supp}\, T_{c} \subset U$ for sufficiently small $c$.  
Assume that there exist points $x_{k} \subset X \setminus U$ such that $x_{k} \in \mathrm{supp}\, T_{-k}$ for any $k \in \mathbb{N}$. 
Let $\{x_{k(j)}\}_{j \in \mathbb{N}}$ be a convergent subsequence and $x = \lim_{j \to \infty}x_{k(j)} \in X \setminus U$.  
Then almost plurisubharmonic function $\varphi$ is unbounded on a neighborhood of $x$, and is not continuous at $x$.  
We have that $x \in \mathrm{supp}\, T$, which gives a contradiction.    
Now we can reduce the proof to the case where $\varphi$ is bounded.  
\end{proof}

\section{Proof of Main results}
The proof of Theorem~\ref{theorem:1} is similar to that of Theorem~\ref{theorem:2}.  
Hence we only prove Theorem~\ref{theorem:2} by induction on $m$.  
The case $m=1$ holds by the arguments in Section~5.  
Assume now that $m \geq 2$ and that the case $m-1$ has already been proved.  
We first note that the cohomology class of $T_{j} + T_{k}$ is in $\mathcal{K}_{NS}$ and 
\[
\mathrm{supp}\, T_{j} \cup \mathrm{supp}\, T_{k} = \mathrm{supp}\, (T_{j} + T_{k}) 
\]
for any $1 \leq j, k \leq m$.  
Put 
$A = \bigcap_{j=1}^{m-1}\mathrm{supp}\, T_{j}$ and 
$B = \mathrm{supp}\, T_{m}$.   
Then Mayer-Vietoris exact sequence yields the commutative diagram 
\[
\xymatrix{
\cdots \ar[r]  & H^{q}(X, F) \ar[r] \ar[d]&
H^{q}(X, F) \oplus H^{q}(X, F)
\ar[r] \ar[d]&
H^{q}(X, F) 
\ar[r] \ar[d]&
\cdots \\
\cdots \ar[r] &  \underset{\substack{A \subset V_{1}\\ B \subset V_{2}}}{\varinjlim}H^{q}(V_{1}\cup V_{2}, F) \ar[r] & 
\underset{A \subset V_{1}}{\varinjlim}H^{q}(V_{1}, F) \oplus \underset{B \subset V_{2}}{\varinjlim}H^{q}(V_{2}, F)
\ar[r] &
\underset{\substack{A \subset V_{1}\\ B \subset V_{2}}}{\varinjlim}H^{q}(V_{1}\cap V_{2}, F)
\ar[r] &
\cdots
}
\]
where the rows are exact sequences.  
We have that 
\[
\underset{\substack{A \subset V_{1}\\ B \subset V_{2}}}{\varinjlim}H^{q}(V_{1}\cup V_{2}, F)  \simeq 
\underset{\bigcap_{j=1}^{m-1}(\mathrm{supp}\, (T_{j} + T_{m})  \subset V}{\varinjlim} H^{q}(V, F) \]
and 
\[
\underset{\substack{A \subset V_{1}\\ B \subset V_{2}}}{\varinjlim}H^{q}(V_{1}\cap V_{2}, F)
\simeq 
\underset{\bigcap_{j=1}^{m}\mathrm{supp}\, T_{j} \subset V}{\varinjlim} H^{q}(V, F).  
\]
Then we complete the proof by 
the induction hypothesis and a diagram-chasing argument.  
\qed 

Let $X = {\bf P}^{n}$ be the complex projective space of dimension $n \geq 3$.  
Then any non-zero closed positive current of type $(1, 1)$ belongs to $\mathcal{K}_{NS}$ 
and we can apply Theorem~\ref{theorem:2} to this case.

\vspace{5mm}

\par\noindent{\scshape \small
Department of Mathematics, \\
Ochanomizu University,  \\
2-1-1 Otsuka, Bunkyo-ku, Tokyo (Japan) }
\par\noindent{\ttfamily chiba.yusaku@ocha.ac.jp}
\end{document}